\documentclass[amssymb,amstex,amsmath,11pt]{article}
\usepackage[margin=1.4in]{geometry}
\usepackage{amssymb,amsmath}              
\usepackage{amsmath,amssymb,latexsym,amsfonts,url}
\usepackage{graphicx}
\usepackage{amsthm}
\usepackage{amssymb}
\usepackage{color}
\usepackage{epstopdf}
\renewcommand{\epsilon}{\varepsilon}

\newtheorem{theorem}{Theorem}[section]
\newtheorem{lemma}[theorem]{Lemma}

\newtheorem{thm}[theorem]{Theorem}

\newtheorem*{theorem*}{Theorem}

\def\liml{\lim\limits}
\def\suml{\sum\limits}
\def\intl{\int\limits}

\def\eps{\varepsilon}
\def\over{\overline}

\def\intl{\int\limits}
\def\suml{\sum\limits}

\def\({\left(}
\def\){\right)}

\def\suml{\sum\limits}

\def\liml{\lim\limits}
\def\intl{\int\limits}

\def\grad{\triangledown}
\newcommand{\Ric}{{\rm Ric}}

\begin{document}
\begin{title}
{Stable minimal hypersurfaces in a Riemannian manifold with pinched negative sectional curvature}
\end{title}
\begin{author}{Nguyen Thac Dung~\thanks{The first author was partially supported by NAFOSTED.} \and Keomkyo Seo~\thanks{The second author was supported by Basic Science Research Program through the
National Research Foundation of Korea(NRF) funded by the Ministry of
Education, Science and Technology(20110005520).}}\end{author}

\date{\today}

\maketitle

\begin{abstract}
We give an estimate of the first eigenvalue of the Laplace operator on a complete noncompact stable minimal hypersurface $M$ in a complete simply connected Riemannian manifold with pinched negative sectional curvature. In the same ambient space, we prove that if a complete minimal hypersurface $M$ has sufficiently small total scalar curvature then $M$ has only one end. We also obtain a vanishing theorem for $L^2$ harmonic $1$-forms on minimal hypersurfaces in a Riemannian manifold with sectional curvature bounded below by a negative constant. Moreover we provide sufficient conditions for a minimal hypersurface in a Riemannian manifold with nonpositive sectional curvature to be stable.\\


\end{abstract}

\section{Introduction}

Let $M$ be a complete noncompact Riemannian manifold and let $\Omega$ be a compact domain in $M$. Let $\lambda_1 (\Omega)>0$ denote the first eigenvalue of the Dirichlet boundary value problem
$$\left\{\begin{array}{llll} \Delta f + \lambda f = 0&\mbox{\ in\ } \mbox{\ $\Omega$\ } \\
 \ f=0 &\mbox{\ on\ } \mbox{\ $\partial \Omega$ } \end{array} \right. $$
where $\Delta$ denotes the Laplace operator on $M$. Then the first eigenvalue $\lambda_1 (M)$ is defined by
$$\lambda_1 (M)= \inf_\Omega \lambda_1 (\Omega),$$
where the infimum is taken over all compact domains in $M$.

In this article we shall find an esimate of the first eigenvalue of the Laplace operator on minimal hypersurfaces. Let us briefly introduce previous results in this direction. In \cite{CL}, Cheung and Leung obtained the first eigenvalue estimate for a complete noncompact submanifold with bounded mean curvature in hyperbolic space. In particular, they proved that for a $n$-dimensional complete minimal submanifold $M$ in the $m$-dimensional hyperbolic space $\mathbb{H}^m$
\begin{eqnarray*}
\frac{1}{4}(n-1)^2 \leq \lambda_1 (M) .
\end{eqnarray*}
Here this inequality is sharp because equality holds when $M$ is
totally geodesic (\cite{McKean}). Bessa and Montenegro \cite{BM} extended this result to complete noncompact submanifolds in a complete simply connected Riemannian manifold with sectional curvature bounded above by a negative constant. Indeed they proved

\begin{theorem*}[\cite{BM}]
Let $N$ be an $n$-dimensional complete simply connected Riemannian manifold with sectional curvature $K_N$ satisfying $K_N \leq -a^2 < 0$ for a positive constant $a>0$. Let $M$ be a an $m$-dimensional complete noncompact submanifold with bounded mean curvature $H$ in $N$ satisfying $|H| \leq b < (m-1)a$. Then
\begin{eqnarray*}
\frac{[(m-1)a-b]^2}{4} \leq \lambda_1 (M) .
\end{eqnarray*}
\end{theorem*}
\noindent Candel \cite{Candel} gave an upper bound for the first eigenvalue of the universal cover of a complete stable minimal surface in the $3$-dimensional hyperbolic space $\mathbb{H}^3$. More precisely, it was proved
\begin{theorem*}[\cite{Candel}]
Let $M$ be a complete simply connected stable minimal surface
in $\mathbb{H}^3$. Then the first eigenvalue
of $M$ satisfies
\begin{eqnarray*}
\frac{1}{4} \leq \lambda_1 (M) \leq \frac{4}{3} .
\end{eqnarray*}
\end{theorem*}
\noindent Recall that an $n$-dimensional minimal hypersurface $M$ in a Riemannian manifold $N$ is called {\it stable} if it holds that for any compactly supported Lipschitz function $f$ on $M$
\begin{eqnarray*}
\int_\Sigma |\nabla f|^2 - (\overline{\Ric}(e_{n+1}) + |A|^2 ) f^2 \geq 0 ,
\end{eqnarray*}
where $e_{n+1}$ is the unit normal vector of $M$ in $N$, $\overline{\Ric}(e_{n+1})$ is the Ricci curvature of $N$ in the direction of $e_{n+1}$, and $|A|$ is the length of the second fundamental form of $\Sigma$. Recently the second author \cite{Seo2011} extended the above theorem to higher-dimensional cases as follows:
\begin{theorem*}[\cite{Seo2011}]
Let $M$ be a complete stable minimal hypersurface in
$\mathbb{H}^{n+1}$ with finite $L^2$-norm of the second fundamental form $A$ (i.e., $\int_M |A|^2 dv  < \infty$). Then we have
\begin{eqnarray*}
\frac{(n-1)^2}{4} \leq  \lambda_1(M) \leq n^2 .
\end{eqnarray*}
\end{theorem*}

In Section 2, we generalize this theorem to a stable minimal hypersurface in a Riemannian manifold with pinched negative sectional curvature, i.e., all sectional curvatures lie between two negative constants. In Section 3, we investigate a vanishing theorem for $L^2$ harmonic $1$-forms on a minimal hypersurface. Miyaoka \cite{Miyaoka} proved that if $M$ is a complete stable minimal hypersurface in Euclidean space, then there is no nontrivial $L^2$ harmonic $1$-form on $M$. Later Yun \cite{Yun} proved that if $M \subset \mathbb{R}^{n+1}$ is a complete minimal hypersurface with sufficiently small total scalar curvature $\int_M |A|^n dv$, then there is non nontrivial $L^2$ harmonic $1$-form on $M$. Yun's result still holds for any complete minimal submanifold with sufficiently small total scalar curvature in hyperbolic space (\cite{Seo2010-1}). Recently the second author \cite{Seo2010} showed that if $M$ is an $n$-dimensional complete stable minimal hypersurface in hyperbolic space satisfying $(2n-1)(n-1) < \lambda_1 (M)$, then there is no nontrivial $L^2$ harmonic $1$-form on $M$. We generalize this result to a complete noncompact Riemannian manifold with sectional curvature bounded below by a nonpositive constant.(See Theorem \ref{l2}.)

One of important results about the geometric structure of a stable minimal hypersurface $M$ in $(n+1)$-dimensional Euclidean space, $n \geq 3$ is that such $M$ must have only one end (\cite{CSZ}). Later Ni \cite{Ni} proved that an $n$-dimensional complete minimal submanifold $M$ in Euclidean space has sufficiently small total scalar curvature, then $M$ must have only one end. More precisely, he proved
\begin{theorem*} {\rm(\cite{Ni})}
Let $M$ be an $n$-dimensional complete immersed minimal
submanifold in $\mathbb{R}^{n+p}$, $n \geq 3$. If
\begin{eqnarray*}
\Big(\int_M |A|^n dv \Big)^{\frac{1}{n}} < C_1 = \sqrt{\frac{n}{n-1}C_s^{-1}},
\end{eqnarray*}
then $M$ has only one end. (Here $C_s$ is a Sobolev constant in
\cite{HS}.)
\end{theorem*}
In Section 4, we shall prove that the analogue of this theorem is still true in a Riemannian manifold with pinched negative sectional
curvature.(See Theorem \ref{thm:one end}.) In Section 5, we provide two sufficient conditions for complete minimal hypersurfaces in a Riemannian manifold with sectional curvature bounded above by a nonpositive
 constant.
\section{Estimates for the bottom of the spectrum}
Let $M$ be an oriented $n-$dimensional manifold immersed in an oriented $(n+1)$-dimensional Riemannian manifold $N$. We choose a local vector field of orthonormal frames $e_1, ..., e_{n+1}$ in $N$ such that, restricted to $M$, the vectors $e_1, ..., e_n$ are tangent to $M$. With respect to this frame field of $N$, let $K_{ijkl}$ be a curvature tensor of $N$. We denote by $K_{ijkl;m}$ the covariant derivative of $K_{ijkl}$. Then we have the following estimate for the bottom of the spectrum of the Laplace operator on a stable minimal hypersurface.
\begin{theorem}\label{t1}
Let $N$ be an $(n+1)$-dimensional complete simply connected Riemannian manifold with sectional curvature satisfying that
$$ K_1\leq K_N\leq K_2, $$
where $K_1, K_2$ are constants and $K_1\leq K_2<0$. Let $M$ be a complete stable non-totally geodesic minimal hypersurface in $N$. Assume that
$$ \lim_{R\rightarrow \infty}R^{-2}\intl_{B(R)}|A|^2=0, $$
where $B(R)$ is a geodesic ball of radius $R$ on $M$. If $ |\grad K|^2=\suml_{i,j,k,l,m}K^2_{ijkl;m}\leq K_3^2|A|^2 $ for some constant $K_3>0$, then we have
\begin{eqnarray*}
-K_2 \dfrac{(n-1)^2}{4}\leq\lambda_1(M) \leq \dfrac{(2K_3-n(K_1+K_2))n}{2}.
\end{eqnarray*}
\end{theorem}
\begin{proof}
By the result of Bessa-Montenegro \cite{BM}, the first eigenvalue $\lambda_1(M)$ of a complete minimal hypersurface $M$ in $N$ is bounded below by $-K_2 \dfrac{(n-1)^2}{4} >0$ as mentioned in the introduction. In the rest of the proof we shall find the upper bound.

Consider a geodesic ball $B(R)$ of radius $R$ centered at $p\in M$. From the definition of $\lambda_1(M)$ and the domain monotonicity of eigenvalue, it follows
\begin{eqnarray} \label{ineq:lambda_1}
\lambda_1(M)\leq\lambda_1(B(R))\leq\frac{\intl_{B(R)}|\grad \phi|^2}{\intl_{B(R)}\phi^2}
\end{eqnarray}
for any compactly supported nonconstant Lipschitz function $\phi$. Choose a test function $g$
satisfying $0 \leq g \leq 1$, $g \equiv 1$ on
$B(R)$, $g \equiv 0 $ on $M \setminus B(2R)$, and
$\displaystyle{|\nabla g| \leq \frac{1}{R}}$. Put $\phi=|A|g$ in the above inequality (\ref{ineq:lambda_1}). Then
$$\begin{aligned} \lambda_1(M)\intl_{B(R)}|A^2|g^2
&\leq \intl_{B(R)}|\grad(|A|g)|^2\\
&=\intl_{B(R)}g^2|\grad|A||^2+\intl_{B(R)}|A|^2|\grad g|^2+2\intl_{B(R)}g|A|\left\langle \grad|A|, \grad g\right\rangle .
\end{aligned}$$
By using Young's inequality, for any $\eps>0$, we see that
$$ 2\intl_{B(R)}g|A|\left\langle\grad|A|, \grad g\right\rangle\leq\frac{1}{\eps}\intl_{B(R)}|A|^2|\grad g|^2+\eps\intl_{B(R)}g^2|\grad|A||^2.  $$
Hence
\begin{equation}\label{e28}
\lambda_1(M)\intl_{B(R)}g^2|A|^2\leq(1+\eps)\intl_{B(R)}g^2|\grad|A||^2+\left(1+\frac{1}{\eps}\right)\intl_{B(R)}|A|^2|\grad g|^2.
\end{equation}
On the other hand, from the equations (1.22) and (1.27) in \cite{SSY}, the length of the second fundamental form $|A|$ satisfies that
$$ |A|\Delta|A|+|\grad|A||^2\geq\sum h_{ijk}^2-2K_3|A|^2+n(2K_2-K_1)|A|^2-|A|^4 $$
or equivalently,
$$ |A|\Delta|A|+2K_3|A|^2-n(2K_2-K_1)|A|^2+|A|^4\geq\sum h_{ijk}^2-|\grad|A||^2 .$$
Since $K_2-K_1\geq0$, this inequality implies
$$ |A|\Delta|A|+2K_3|A|^2-nK_2|A|^2+|A|^4\geq\sum h_{ijk}^2-|\grad|A||^2 =|\nabla A|^2 -|\nabla |A||^2 .$$
Applying the following Kato-type inequality \cite{Xin}
$$ |\grad A|^2-|\grad|A||^2\geq\frac{2}{n}|\grad|A||^2, $$
we obtain
$$ |A|\Delta|A|+2K_3|A|^2-nK_2|A|^2+|A|^4\geq\frac{2}{n}|\grad|A||^2 $$
Multiplying both side by a Lipschitz function $f^2$ with compact support in $B(R)\subset M$ and integrating over $B(R)$, we get
$$ \intl_{B(R)}f^2|A|\Delta|A|+2K_3\intl_{B(R)}f^2|A|^2-nK_2\intl_{B(R)}f^2|A|^2+\intl_{B(R)}f^2|A|^4\geq\frac{2}{n}\intl_{B(R)}f^2|\grad|A||^2. $$
The divergence theorem gives
$$ 0=\intl_{B(R)}\text{div}(|A|f^2\grad|A|) =\intl_{B(R)}f^2|A|\Delta|A|+\intl_{B(R)}f^2|\grad|A||^2+2\intl_{B(R)}f|A|\left\langle\grad|A|, \grad f\right\rangle $$
Therefore
\begin{align}
\intl_{B(R)}f^2|A|^4+(2K_3-nK_2)\intl_{B(R)}f^2|A|^2
&\geq\frac{2}{n}\intl_{B(R)}f^2|\grad|A||^2+\intl_{B(R)}f^2|\grad|A||^2\notag\\
&{~~~~~~~}+2\intl_{B(R)}f|A|\left\langle\grad|A|, \grad f\right\rangle\label{e211}
\end{align}
Since $M$ is stable, for any compactly supported function $\phi$ on $M$,
\begin{eqnarray} \label{ineq:stability}
\intl_M|\grad\phi|^2-\Big(|A|^2+\over{\text{Ric}}(e_{n+1})\Big)\phi^2\geq0 .
\end{eqnarray}
The assumption on sectional curvature of $N$ implies that
$$ nK_1\leq \over{\text{Ric}}(e_{n+1})=R_{n+1, 1, n+1, 1}+\ldots+ R_{n+1, n, n+1, n}\leq nK_2 .$$
Thus the stability inequality (\ref{ineq:stability}) becomes
 $$ \intl_M|\grad\phi|^2-\Big(|A|^2+nK_1\Big)\phi^2\geq0.$$
Replacing $\phi$ by $|A|f$ in the above inequality, we get
$$ \intl_{B(R)}|\grad(f|A|)|^2\geq \intl_{B(R)}f^2|A|^4+nK_1\intl_{B(R)}f^2|A|^2, $$
namely,
\begin{align}
\intl_{B(R)}|A|^2|\grad f|^2+\intl_{B(R)}&f^2|\grad|A||^2+2\intl_{B(R)}f|A|\left\langle\grad|A|, \grad f\right\rangle\notag\\
&{~~~~}\geq \intl_{B(R)}f^2|A|^4+nK_1\intl_{B(R)}f^2|A|^2 . \label{e212}
\end{align}
Combining the inequalities \eqref{e211} and \eqref{e212} gives
\begin{equation}\label{e213}
\intl_{B(R)}|A|^2|\grad f|^2+\Big(2K_3-n(K_1+K_2)\Big)\intl_{B(R)}f^2|A|^2\geq\frac{2}{n}\intl_{B(R)}f^2|\grad|A||^2 .
\end{equation}
Substituting $g$ for $f$ in the inequality  \eqref{e213} and using the inequality \eqref{e28}, we obtain
\begin{align*}
\Big\{1+\frac{\Big(2K_3-n(K_1+K_2)\Big)}{\lambda_1(M)}\Big(1&+\frac{1}{\eps}\Big)\Big\}\intl_{B(R)}|A|^2|\grad g|^2 \\
&\geq\left\{\frac{2}{n}-\frac{\Big(2K_3-n(K_1+K_2)\Big)}{\lambda_1(M)}(1+\eps)\right\}\intl_{B(R)}g^2|\grad|A||^2 . 
\end{align*}
Suppose that $\lambda_1(M)>\dfrac{\Big(2K_3-n(K_1+K_2)\Big)n}{2}$. Choose a sufficiently samll $\eps>0$ satisfying that
 $$\left\{\frac{2}{n}-\frac{\Big(2K_3-n(K_1+K_2)\Big)}{\lambda_1(M)}(1+\eps)\right\}>0 .$$
\noindent Using the fact that $|\grad g|\leq\dfrac{1}{R}$ and growth condition on $\intl_{B(R)} |A|^2$, we can conclude that by letting $R\to\infty$
$$ \intl_M|\grad|A||^2=0 .$$
This implies that $|A|$ is constant. Since the volume of $M$ is infinite, one sees that $|A|\equiv0$. This means that $M$ is totally geodesic, which is impossible by our assumption. Therefore we have $\lambda_1(M)\leq\dfrac{\Big(2K_3-n(K_1+K_2)\Big)n}{2} $ .
\end{proof}


In particular, if $N$ is the $(n+1)$-dimensional hyperbolic space $\mathbb{H}^{n+1}$, then one sees that $K_1=K_2=-1$, and hence $|\grad K|^2=0$, i.e., $K_3=0$. Moreover, as mentioned in the introduction, it follows from the McKean's result \cite{McKean} that the first eigenvalue $\lambda_1(M)$ of a complete totally geodesic hypersurface $M \subset \mathbb{H}^{n+1}$ satisfies $\lambda_1(M)=\frac{(n-1)^2}{4}$. Therefore one can recover the following theorem which was proved by the second author.
\begin{theorem}[\cite{Seo2011}]
Let $M$ be a complete stable minimal hypersurface in $\mathbb{H}^{n+1}$ with $\int_M |A|^2dv<\infty$. Then we have
\begin{eqnarray*}
\frac{(n-1)^2}{4}\leq\lambda_1(M)\leq n^2 .
\end{eqnarray*}
\end{theorem}

\section{Vanishing theorem on minimal hypersufaces}
We first recall some useful results which we shall use in this section.
\begin{lemma} [\cite{Leung}] \label{lem:Leung}
Let $M$ be an $n$-dimensional complete immersed minimal hypersurface in a Riemannian manifold $N$. If all the sectional curvatures of $N$ are bounded below by a constant $K$ then
$$ \text{Ric}\geq (n-1)K-\frac{n-1}{n}|A|^2. $$
\end{lemma}
\noindent For harmonic $1$-forms, one has the Kato-type inequality as follows:
\begin{lemma} [\cite{Wang}]
Let $\omega$ is a harmonic $1$-form on an $n$-dimensional Riemannian manifold $M$, then
\begin{eqnarray}
 |\grad\omega|^2-|\grad|\omega||^2\geq\frac{1}{n-1}|\grad|\omega||^2 . \label{ineq:Kato}
\end{eqnarray}
\end{lemma}
\noindent It is well known that the following Sobolev inequality on a minimal submanifold in a nonpositively curved manifold.
\begin{lemma}[\cite{HS}]
Let $M^n$ be a complete immersed minimal submanifold in a nonpositively curved manifold
${N}^{n+p}$, $n \geq 3$. Then for any $\phi \in W_0
^{1,2}(M)$ we have
\begin{eqnarray}
\Big(\int_M |\phi|^{\frac{2n}{n-2}} dv\Big)^{\frac{n-2}{n}} \leq C_s
\int_M |\nabla \phi|^2 dv,  \label{ineq:HS}
\end{eqnarray}
where $C_s$ is the Sobolev constant which dependents only on $n \geq 3$.
\end{lemma}

\noindent If $M\subset \mathbb{H}^{n+1}$ is a complete stable minimal hypersurface with large first eigenvalue, then we have the nonexistence theorem of $L^2$ harmonic $1$-forms on $M$. More precisely, we have
\begin{theorem*}[\cite{Seo2010}]
Let $M$ be a complete stable minimal hypersurface in $\mathbb{H}^{n+1}$ satisfying that
\begin{eqnarray*}
(2n-1)(n-1) < \lambda_1 (M) .
\end{eqnarray*}
Then there are no nontrivial $L^2$ harmonic 1-forms on $M$.
\end{theorem*}
 In this section, we shall generalize the above theorem to a complete stable minimal hypersurface in a Riemannian manifold with sectional curvature bounded below by a nonpositive constant as follows:
\begin{theorem}\label{l2}
Let $N$ be $(n+1)$-dimensional Riemannian manifold with sectional curvature $K_N$ satisfying $K\leq K_N$
where $K \leq 0$ is a constant. Let $M$ be a complete noncompact stable non-totally geodesic minimal hypersurface in $N$. Assume that
$$ -K(2n-1)(n-1)<\lambda_1(M) $$
Then there are no nontrivial $L^2$ harmonic 1-forms on $M$.
\end{theorem}
\begin{proof}
Let $\omega$ be an $L^2$ harmonic $1$-form on M, i.e.,
\begin{eqnarray*}
\Delta \omega = 0 \ \ {\rm and}\ \  \int_M |\omega|^2 dv < \infty .
\end{eqnarray*}
In an abuse of notation, we will refer to a harmonic $1$-form and its dual harmonic vector field both by $\omega$. From Bochner formula, it follows
$$ \Delta|\omega|^2=2(|\grad\omega|^2+\text{Ric}(\omega, \omega)) .$$
On the other hand, one sees that
$$ \Delta|\omega|^2=2(|\omega|\Delta|\omega|+|\grad|\omega||^2). $$
Thus we obtain
$$ |\omega|\Delta|\omega|-\text{Ric}(\omega, \omega)=|\grad\omega|^2-|\grad|\omega||^2. $$
Applying Lemma \ref{lem:Leung} and Kato-type inequality (\ref{ineq:Kato}) yields
\begin{equation}\label{e21k} |\omega|\Delta|\omega|+\frac{n-1}{n}|A|^2|\omega|^2-(n-1)K|\omega|^2\geq\frac{1}{n-1}|\grad|\omega||^2.
\end{equation}
The stability of $M$ implies that
$$ \intl_M\Big(|\grad\phi|^2-(|A|^2+\over{\text{Ric}}(e_{n+1}))\phi^2 \Big)\geq0 $$
for any compactly supported Lipschitz function $\phi$ on $M$. Since $nK\leq \over{\text{Ric}}(e_{n+1})$, we have
$$ \intl_M\Big(|\grad\phi|^2-(|A|^2+nK)\phi^2\Big)\geq0 $$
 Replacing $\phi$ by $|\omega|\phi$ gives
$$ \intl_M\Big(|\grad(|\omega|\phi)|^2-(|A|^2+nK)(|\omega|\phi)^2\Big)\geq0 .$$
Applying the divergence theorem, we get
 \begin{align}
0&\leq -\intl_M|\omega|\phi\Delta(|\omega|\phi)-\intl_M(|A|^2+nK)|\omega|^2\phi^2\notag\\
&=-\intl_M\phi|\omega|^2\Delta\phi-\intl_M\phi^2(|\omega|\Delta|\omega|+|A|^2|\omega|^2)
-2\intl_M\phi|\omega|\left\langle\grad|\omega|, \grad\phi\right\rangle-nK\intl_M\phi^2|\omega|^2\notag\\
&=\intl_M\left\langle\grad(\phi|\omega|^2), \grad\phi \right\rangle -\intl_M\phi^2(|\omega|\Delta|\omega|+|A|^2|\omega|^2)-2\intl_M\phi|\omega|\left\langle\grad|\omega|, \grad\phi\right\rangle-nK\intl_M\phi^2|\omega|^2\notag\\
&=\intl_M|\omega|^2|\grad\phi|^2-\intl_M\phi^2(|\omega|\Delta|\omega|+|A|^2|\omega|^2)-nK\intl_M\phi^2|\omega|^2\label{e22k}
\end{align}
Using the inequalities \eqref{e21k} and \eqref{e22k}, we obtain
\begin{align}
0\leq\intl_M|\omega|^2|\grad\phi|^2-(2n-1)K\intl_M\phi^2|\omega|^2-\frac{1}{n-1}\intl_M\phi^2|\grad|\omega||^2-\frac{1}{n}\intl_M|A|^2\phi^2|\omega|^2\label{e23k}
\end{align}
From the definition of the bottom of the spectrum, it follows
\begin{align}
\lambda_1(M)\intl_M\phi^2|\omega|^2
&\leq\intl_M|\grad(\phi|\omega|)|^2\notag\\
&=\intl_M|\omega|^2|\grad\phi|^2+\phi^2|\grad|\omega||^2+2|\omega|\phi\left\langle \grad\phi, \grad|\omega|\right\rangle \notag\\
&\leq\left(1+\frac{1}{\eps}\right)\intl_M|\omega|^2|\grad\phi|^2+(1+\eps)\intl_M|\grad|\omega||^2\phi^2 , \label{e24k}
\end{align}
where we used Schwarz inequality and Young's inequality for $\eps>0$ in the last inequality.
Combining the inequalities \eqref{e23k} and \eqref{e24k}, we have
\begin{align*}
0\leq &\left\{\left(1+\frac{1}{\eps}\right)\frac{(2n-1)(-K)}{\lambda_1}+1\right\}\intl_M|\omega|^2|\grad\phi|^2\\
&+\left\{(1+\eps)\frac{(2n-1)(-K)}{\lambda_1(M)}-\frac{1}{n-1}\right\}\intl_M|\grad|\omega||^2\phi^2-\frac{1}{n}\intl_M|A|^2|\omega|^2\phi^2.
\end{align*}
Now fix a point $p\in M$ and consider a geodesic ball $B_p(R)$ of radius $R$ centered at $p$. Choose a test function $\phi$
satisfying that $0 \leq \phi \leq 1$, $\phi \equiv 1$ on $B_p(R)$, $\phi \equiv 0 $ on $M \setminus B_p(2R)$, and
$\displaystyle{|\nabla \phi| \leq \frac{1}{R}}$.  Since $\lambda_1(M)>-K(2n-1)(n-1)$, one can choose a sufficiently small $\eps>0$ such that
\begin{eqnarray*}
(1+\eps)\frac{(2n-1)(-K)}{\lambda_1(M)}-\frac{1}{n-1}<0 .
\end{eqnarray*}
Letting $R\to\infty$ and using the fact that $\intl_M |\omega|^2 < \infty$, we finally obtain
$$ \intl_M|\grad|\omega||^2\phi^2= \intl_M|A|^2|\omega|^2\phi^2=0 ,$$
which implies that either $\omega\equiv0$ or $|A|\equiv0$. Since $M$ is not totally geodesic by the assumption, we see that $\omega\equiv 0$. Therefore we get the conclusion.
\end{proof}

\section{Rigidity of minimal hypersurfaces}

Let $u$ be a harmonic function on $M$. Bochner's formula says
$$ \frac{1}{2}\Delta(|\grad u|^2)=\suml_{i,j=1}^nu_{ij}^2+\text{Ric}(\grad u, \grad u). $$
Since
$$ \frac{1}{2}\Delta(|\grad u|^2)=|\grad u|\Delta(|\grad u|)+|\grad|\grad u||^2, $$
we obtain
$$ |\grad u|\Delta(|\grad u|)+|\grad|\grad u||^2 =\suml_{i,j=1}^nu_{ij}^2+\text{Ric}(\grad u, \grad u).$$
Using Lemma \ref{lem:Leung} and Kato-type inequality we get
\begin{equation}\label{e22s4}
|\grad u|\Delta|\grad u|+\frac{n-1}{n}|A|^2|\grad u|^2-(n-1)K_1|\grad u|^2\geq\frac{1}{n-1}|\grad|\grad u||^2.
\end{equation}
By using the same arguments as in the proof of Theorem \ref{l2} in Section 3, we shall prove that if $M$ is an $n$-dimensional complete noncompact stable minimal hypersurface in a Riemannian manifold with sectional curvature bounded below by a nonpositive constant, then $M$ has only one end. More precisely, we prove
\begin{theorem}
Let $N$ be $(n+1)$-dimensional Riemannian manifold with sectional curvature $K_N$ satisfying $K\leq K_N$
where $K \leq 0$ is a constant. Let $M$ be a complete noncompact stable non-totally geodesic minimal hypersurface in $N$. Assume that
$$ -K(2n-1)(n-1)<\lambda_1(M) $$
Then $M$ must have only one end.
\end{theorem}
\begin{proof}Suppose that $M$ has at least two ends. If $M$ has more than one end then there exists a non-trivial harmonic function $u$ on $M$ with finite total energy.(\cite{Wei}, see also \cite{CSZ}.) Then the harmonic function $u$ satisfies the inequality \eqref{e22s4}. Applying the same arguments as in the proof of Theorem \ref{l2}, we conclude that $|\grad u|=0$. Hence $u$ is constant, which is a contradiction. Therefore $M$ should have only one end.
\end{proof}

Even without the stability condition on a complete minimal hypersurface we have a gap theorem for minimal hypersurfaces with finite total scalar curvature as follows:
\begin{theorem} \label{thm:one end}
Let $N$ be an $(n+1)$-dimensional Riemannian manifold with sectional curvature satisfying
$$ K_1\leq K_N\leq K_2, $$
where $K_1, K_2$ are constants and $K_1\leq K_2<0$. Let $M$ be a complete minimal hypersurface in $N$. If
$$ \left(\int_M|A|^n\right)^{\frac{1}{n}}\leq\frac{1}{n-1}\sqrt[]{\frac{n(nK_2-4K_1)}{K_2}C_s^{-1}} $$
for $n>4\frac{K_1}{K_2}$, then $M$ has only one end. (Here $C_s$ is the Sobolev constant in \cite{HS}.)
\end{theorem}
\begin{proof}Suppose that $M$ has at least two ends. From Wei's result \cite{Wei}, it follows that if $M$ has more than one end then there exists a non-trivial harmonic function $u$ on $M$ with finite total energy. Let $\phi$ be a compactly supported Lipschitz function on $M$. Put $f:=|\grad u|$ in the inequality \eqref{e22s4}. Multiplying both sides of the inequality \eqref{e22s4} by $\phi^2$ and integrating over $M$, we obtain
$$ \intl_M\phi^2f\Delta f+\frac{n-1}{n}\intl_M\phi^2|A|^2f^2-(n-1)K_1\intl_M\phi^2f^2\geq\frac{1}{n-1}\intl_M\phi^2|\grad f|^2. $$
By the definition of $\lambda_1(M)$ and Theorem \ref{t1}, we have
$$ \frac{(n-1)^2}{4}(-K_2)\leq\lambda_1(M)\leq\frac{\int_M|\grad(\phi f)|^2}{\int_M\phi^2f^2}. $$
Hence
$$\intl_M\phi^2f\Delta f+\frac{n-1}{n}\intl_M\phi^2|A|^2f^2+\frac{4K_1}{(n-1)K_2}\intl_M|\grad(\phi f)|^2\geq\frac{1}{n-1}\intl_M\phi^2|\grad f|^2.$$
Integration by parts gives
 \begin{align*}
{}&-\intl_M|\grad f|^2\phi^2-2\intl_Mf\phi\left\langle\grad f, \grad\phi\right\rangle+ \frac{n-1}{n}\intl_M\phi^2|A|^2f^2\\
&{~~~~~}+\frac{4K_1}{(n-1)K_2}\left(\intl_Mf^2|\grad\phi |^2+\intl_M\phi^2|\grad f|^2+2\intl_Mf\phi\left\langle\grad f, \grad\phi\right\rangle \right)\geq\frac{1}{n-1}\intl_M\phi^2|\grad f|^2 .
\end{align*}
Using  Schwarz inequality, for any $\eps>0$, we have
\begin{align}
{}&\frac{n-1}{n}\intl_M\phi^2|A|^2f^2+\left(\frac{4K_1}{(n-1)K_2}+\frac{1}{\eps}+\frac{4K_1}{(n-1)\eps K_2}\right)\intl_Mf^2|\grad\phi|^2\notag\\
&\geq\left(\frac{n}{n-1}-\eps-\frac{4K_1}{(n-1)K_2}-\frac{4K_1\eps}{(n-1)K_2}\right)\intl_M\phi^2|\grad f|^2 .\label{ehba}
\end{align}
On the other hand, by Schwarz inequality, for any $\eta>0$, we see that
$$ (1+\eta)\intl_M\phi^2|\grad f|^2+\left(1+\frac{1}{\eta}\right)\intl_Mf^2|\grad\phi|^2\geq\intl_M|\grad(f\phi)|^2. $$
From the Sobolev inequality (\ref{ineq:HS}), it follows
\begin{align}
(1+\eta)\intl_M\phi^2|\grad f|^2+\left(1+\frac{1}{\eta}\right)\intl_Mf^2|\grad\phi|^2\geq C_s^{-1}\left(\intl_M(f\phi)^{\frac{2n}{n-2}}\right)^{\frac{n-2}{n}},\label{ehbon}
\end{align}
where $C_s$ is the Sobolev constant.

For a fixed point $p\in M$ and $R>0$, choose a test function $\phi$
satisfying that $0 \leq \phi \leq 1$, $\phi \equiv 1$ on $B_p(R)$, $\phi \equiv 0 $ on $M \setminus B_p(2R)$, and
$\displaystyle{|\nabla \phi| \leq \frac{1}{R}}$, where $B_p(R)$ denotes the geodesic ball of radius $R$ centered at $p\in M$. Define
$$ \Gamma(\eps, R):= \left(\frac{4K_1}{(n-1)K_2}+\frac{1}{\eps}+\frac{4K_1}{(n-1)\eps K_2}\right)\intl_Mf^2|\grad\phi|^2$$
and
$$ c(\eps):= \left(\frac{nK_2-4K_1-4K_1\eps}{(n-1)K_2}-\eps\right) .$$
Note that $c(\eps)>0$ for a sufficiently small $\eps >0$. Since $u$ has finite total energy, $f=|\grad u|$ has a finite $L^2$-norm. Therefore
$$ \liml_{R\to\infty}\Gamma(\eps, R)=0 $$
for any fixed $\eps>0$.

\noindent Choose a sufficiently small $\eps>0$. Then from inequalities \eqref{ehba} and \eqref{ehbon} it follows
\begin{align}
\frac{n-1}{n}\intl_M|A|^2f^2\phi^2+\Gamma(\eps, R)
&\geq c(\eps)\intl_M\phi^2|\grad f|^2\notag\\
&\geq\frac{c(\eps)C_s^{-1}}{\eta+1}\left(\intl_M(f\phi)^{\frac{2n}{n-2}}\right)^{\frac{n-2}{n}}-\frac{c(\eps)}{\eta}\intl_Mf^2|\grad \phi|^2 . \label{e25}
\end{align}
Using H\"{o}lder inequality, we get
\begin{align}
\intl_M|A|^2\phi^2f^2\leq\left(\intl_M|A|^n\right)^{\frac{2}{n}}\left(\intl_M(f\phi)^{\frac{2n}{n-2}}\right)^{\frac{n-2}{n}}. \label{ineq:Holder}
\end{align}
Combining these inequalities \eqref{e25} and \eqref{ineq:Holder} gives
$$ \Gamma(\eps, R)+\frac{c(\eps)}{\eta}\intl_Mf^2|\grad \phi|^2\geq\left(\frac{c(\eps)C_s^{-1}}{\eta+1}-\frac{n-1}{n}\left(\intl_M|A|^n\right)^{\frac{2}{n}} \right)\left(\intl_M(f\phi)^{\frac{2n}{n-2}}\right)^{\frac{n-2}{n}} .$$
Now choose a sufficiently small $\eta>0$. Then the assumption on the $L^n$-norm of the second fundamental form $A$ guarantees that
$$ \frac{c(\eps)C_s^{-1}}{\eta+1}-\frac{n-1}{n}\left(\intl_M|A|^n\right)^{\frac{2}{n}}>0 . $$
Letting $R\to\infty$, we obtain that $f=|\grad u|=0$, i.e., $u$ is constant. This is a contradiction. Hence we conclude that $M$ should have only one end.
\end{proof}

\section{Stability of minimal hypersurfaces}
In this section we prove that if the second fundamental form $|A|$ at each point or the total scalar curvature $\int_M |A|^n$ is sufficiently small in a complete minimal
hypersurface $M$ in a nonpositively curved Riemannian manifold, then $M$ must be stable.
\begin{theorem}
Let $N$ be $(n+1)$-dimensional Riemannian manifold with sectional curvature $K_N$ satisfying $K_N\leq K\leq 0$
where $K\leq 0$ is a constant. Let $M$ be a complete stable minimal hypersurface in $N$. If
$$|A|^2\leq -K\frac{(n+1)^2}{4}$$
then $M$ is stable.
\end{theorem}
\begin{proof}
From the result of Bessa-Montenegro \cite{BM}, we have
$$ -K\frac{(n-1)^2}{4}\leq\lambda_1(M)\leq\frac{\intl_M|\grad f|^2}{\intl_Mf^2} $$
for any compactly supported nonconstant Lipschitz function $f$ on $M$.

The assumption on sectional curvature on $N$ implies that $\over{\text{Ric}}(e_{n+1})\leq nK $. Hence
\begin{align*}
\intl_M \Big\{|\grad f|^2-(|A|^2+\over{\text{Ric}}(e_{n+1}))f^2\Big\}
&\geq \intl_M \left(-K\frac{(n-1)^2}{4}-|A|^2-nK\right)f^2\\
&\geq \intl_M \left(-K\frac{(n-1)^2}{4}+K\frac{(n+1)^2}{4}-nK\right)f^2\\
&=0 ,
\end{align*}
which means that $M$ is stable.
\end{proof}

\begin{thm} \label{thm:sufficient condition 2}
Let $N$ be $(n+1)$-dimensional Riemannian manifold with sectional curvature $K_N$ satisfying $K_N\leq K\leq 0$
where $K \leq 0$ is a constant. Let $M$ be a complete stable minimal hypersurface in $N$. Assume that
$$\intl_M |A|^n  \leq \Big(\frac{1}{C_s}\Big)^{\frac{n}{2}} .$$ Then $M$ is stable.
\end{thm}
\begin{proof}
It suffices to show that
\begin{eqnarray*}
\intl_M \Big\{|\grad f|^2-(|A|^2+\over{\text{Ric}}(e_{n+1}))f^2\Big\} \geq 0
\end{eqnarray*}
for all compactly supported Lipschitz function $f$. Using the Sobolev
inequality (\ref{ineq:Sobolev}), we have
\begin{align} \label{ineq:1}
\intl_M \Big\{|\grad f|^2-(|A|^2+\over{\text{Ric}}(e_{n+1}))f^2\Big\} &\geq \frac{1}{C_s}
\Big(\intl_M |f|^{\frac{2n}{n-2}} \Big)^{\frac{n-2}{n}} - \intl_M
(|A|^2+\over{\text{Ric}}(e_{n+1}))f^2   \nonumber \\
&\geq \frac{1}{C_s}
\Big(\intl_M |f|^{\frac{2n}{n-2}} \Big)^{\frac{n-2}{n}} - \intl_M
|A|^2f^2 ,
\end{align}
where we used the fact that $\over{\text{Ric}}(e_{n+1})\leq nK \leq 0$ in the last inequality.

\noindent
On the other hand, applying H\"{o}lder inequality, we get
\begin{eqnarray} \label{ineq:2}
\intl_M |A|^2 f^2 dv \leq \Big(\intl_M |A|^n  \Big)^{\frac{2}{n}}
\Big(\intl_M |f|^{\frac{2n}{n-2}} \Big)^{\frac{n-2}{n}}.
\end{eqnarray}
Combining (\ref{ineq:1}) with (\ref{ineq:2}) we have
\begin{align*}
\intl_M \Big\{|\grad f|^2-(|A|^2+\over{\text{Ric}}(e_{n+1}))f^2\Big\}&\geq
\Big\{\frac{1}{C_s} - \Big(\intl_M |A|^n \Big)^{\frac{2}{n}}\Big\} \Big(\intl_M |f|^{\frac{2n}{n-2}} \Big)^{\frac{n-2}{n}} \\
 &\geq 0 ,
\end{align*}
which completes the proof.
\end{proof}

\vskip 0.3cm
\noindent
{\bf Acknowledgment: }The first author would like to express his gratitude to NSC, Taiwan for supporting him three year scholarship at National Tsinghua University.

\vskip 1cm
\noindent Nguyen Thac Dung\\
Department of Mathematics, National Tsinghua University, No. 101, Sec. 2, Kuangfu Road, Hsinchu, Taiwan, R.O.C.\\
{\tt E-mail:dungmath@yahoo.co.uk}\\

\bigskip
\noindent Keomkyo Seo\\
Department of Mathematics, Sookmyung Women's University, Hyochangwongil 52, Yongsan-ku, Seoul, 140-742, Korea\\
{\tt E-mail:kseo@sookmyung.ac.kr}\\
\end{document}